\documentclass[12pt]{article}
\usepackage{latexsym}
\usepackage{epsfig,enumerate,bbm}
\usepackage{amsmath,amsthm,amssymb}
\usepackage[active]{srcltx}
\usepackage[usenames]{color}\usepackage{graphicx}
\numberwithin{equation}{section}
\definecolor{brown}{cmyk}{0, 0.72, 1, 0.45}
\definecolor{grey}{gray}{0.5}

\renewcommand{\epsilon}{\varepsilon}

\def\deg{\text{deg}}

\def\cG{{\cal G}}

\newcounter{rot}%\addtocounter{rot}{1}, \therot
\def\leb{\leq_b}

\def\a{\alpha} \def\b{\beta} \def\d{\delta} \def\D{\Delta}
\def\e{\epsilon} \def\f{\phi}   \def\g{\gamma}
   
\def\z{\zeta} \def\th{\theta}    
 \def\m{\mu} \def\n{\nu} \def\p{\pi}
\def\r{\rho}  \def\s{\sigma} 
\def\t{\tau}   \def\Om{\Omega}

\def\cC{{\cal C}}

\newtheorem{remark}{Remark}

\newtheorem*{conjecture*}{Conjecture}
\newtheorem{theorem}{Theorem}[section]
\newtheorem*{theorem*}{Theorem}
\newtheorem{lemma}[theorem]{Lemma}

\newtheorem{bound}{Bound}[section]

\newcommand{\ol}[1]{\overline{#1}}

\newcommand{\rdown}[1]{{\left\lfloor #1\right \rfloor}}

\newcommand{\brac}[1]{\left(#1\right)}

\newcommand{\bfrac}[2]{\left(\frac{#1}{#2}\right)}

\def\cE{{\cal E}}

\newcommand{\set}[1]{\left\{#1\right\}}

\def\E{\mathbb{E}}

\def\Pr{\mathbb{P}}

\newcommand{\ignore}[1]{}

\newcommand{\cA}{{\cal A}}

\newcommand{\card}[1]{\left|#1\right|}

\newcommand{\beq}[1]{\begin{equation}\label{#1}}
\newcommand{\eeq}{\end{equation}}

\newcommand{\Bin}{\operatorname{Bin}}

\newcommand{\proofstart}{{\bf Proof\hspace{2em}}}
\newcommand{\proofend}{\hspace*{\fill}\mbox{$\Box$}}
\DeclareMathOperator{\Cov}{Cov}
\DeclareMathOperator{\Var}{Var}
\title{On the game chromatic number of sparse random graphs}

\author{Alan Frieze\thanks{Research supported in part by NSF Grant CCF1013110. Email: {\tt alan@random.math.cmu.edu}.}
\ and Simcha Haber\thanks{Email: {\tt simi@andrew.cmu.edu}.} \ and  Mikhail Lavrov\thanks{Email:
{\tt mlavrov@andrew.cmu.edu}.}\\
Department of Mathematical Sciences,\\
Carnegie Mellon University,\\
Pittsburgh PA, 15213,\\
USA.
}

\usepackage{pstricks,pst-node}
\newlength{\blremwidth}
\newcounter{blremcounter}

\begin{document}
\maketitle
\begin{abstract}
Given a graph $G$ and an integer $k$, two players take turns coloring the vertices of
$G$ one by one using $k$ colors so that neighboring vertices get different colors.
The first player wins iff at the end of the game all the vertices of $G$ are colored.
The game chromatic number $\chi_g(G)$ is the minimum $k$ for which the first player
has a winning strategy.
The paper \cite{BFS} began the analysis of the asymptotic behavior of this
parameter for a random graph $G_{n,p}$.
This paper provides some further analysis for graphs with constant average degree i.e. $np=O(1)$
and for random regular graphs. We show that w.h.p. $c_1\chi(G_{n,p})\leq \chi_g(G_{n,p})\leq
c_2\chi(G_{n,p})$ for some absolute constants $1<c_1< c_2$. We also prove that if $G_{n,3}$
denotes a random $n$-vertex cubic graph then w.h.p. $\chi_g(G_{n,3})=4$. 

\end{abstract}

\section{Introduction}
Let $G=(V,E)$ be a graph and let $k$ be a positive integer. Consider the following game in which
two players A(lice) and B(ob) take turns in coloring the vertices of $G$ with
$k$ colors. Each move consists of choosing an uncolored vertex of the graph and assigning to it
a color from $\{1, \ldots, k\}$ so that the resulting coloring is {\em proper},
i.e., adjacent vertices get different colors.
A wins if all the vertices of $G$ are eventually colored. B wins if at some point in the
game the current partial coloring cannot be extended to a complete coloring of $G$,
i.e., there is an uncolored vertex such that each of the $k$ colors appears
at least once in its neighborhood. We assume that A goes first (our results will not
be sensitive to this choice). The {\em game chromatic number}
$\chi_g(G)$ is the least integer $k$ for which A has a winning
strategy.

This parameter is well defined, since it is easy to see that A always wins if the number of
colors is larger than the maximum degree of $G$.  Clearly, $\chi_g(G)$ is at least as large as
the ordinary chromatic number $\chi(G)$, but it can be considerably more.
The game was first considered by Brams about 25 years ago in the
context of coloring planar graphs and was described in Martin
Gardner's column \cite{Ga} in Scientific American in 1981. The game remained
unnoticed by the graph-theoretic community until
Bodlaender \cite{Bod} re-invented it. For a survey see
Bartnicki, Grytczuk, Kierstead and Zhu \cite{BGKZ}.

In this paper, we study the game chromatic number of the random graph \( G_{n,p} \)
and the random $d$-regular graph $G_{n,d}$.
Define $b=\frac{1}{1-p}$. The following estimates were proved in Bohman, Frieze and Sudakov \cite{BFS}.
\begin{theorem}\label{th1}\
\begin{description}
\item[(a)] There exists $K>0$ such that for $\e>0$ and $p \geq (\ln n)^{K\e^{-3}} /n$
we have that w.h.p.\footnote{A sequence of events $\cE_n$ occurs {\em with high
probability} (w.h.p.) if $\lim_{n\to\infty}\Pr(\cE_n)=1$}
$$\chi_g(G_{n,p})\geq (1-\e)\frac{n}{\log_bnp}\, . $$
\item[(b)] If $\a>2$ is a constant, \( K=\max\{\frac{2\a}{\a-1},\frac{\a}{\a-2}\} \)
and $p \geq (\ln n)^K/n$ then w.h.p.
$$\chi_g(G_{n,p})\leq \a \frac{n}{\log_bnp}\, .$$
\end{description}
\end{theorem}
In this paper we complement these results by considering the case where $p=\frac{d}{n}$ where
$d$ is at least some sufficiently large constant. We will assume that $d\leq n^{1/4}$ since Theorem \ref{th1}
covers larger $d$.
\begin{theorem}\label{th2}
Let $p=\frac{d}{n}$ where $d$ is larger than some absolute constant and $d\leq n^{1/4}$.
\begin{description}
\item[(a)] If $\a<\frac{4}{7}$ is a constant
then w.h.p.
$$\chi_g(G_{n,p})\geq \frac{\a d}{\ln d}.$$
\item[(b)] If $\a$ is a sufficiently large constant then
w.h.p.
$$\chi_g(G_{n,p})\leq \frac{\a d}{\ln d}.$$
\end{description}
\end{theorem}
Note that when $p=o(1)$ we have $\frac{n}{\log_b np}\sim \frac{d}{\ln d}$. Note also that the bounds in
Theorem \ref{th1} are stronger than those in Theorem \ref{th2}, whenever both results are applicable.

It is natural to compare our bounds with the asymptotic behavior of the
ordinary chromatic number of random graph. It is known by the results of
Bollob\'as \cite{Bo} and {\L}uczak \cite{Lu} that when $p=o(1)$,
$\chi(G_{n,p})=(1+o(1))\frac{d}{2\ln d}$ w.h.p.\ (Of course a stronger result is now known,
see Achlioptas and Naor \cite{AN}). Thus Theorem \ref{th2}
shows that the game chromatic number of $G_{n,p}$
is at most (roughly) twelve times and at least (roughly) 8/7 times its chromatic number.

Having proved Theorem \ref{th2}, we extend the results to the random $d$-regular graph $G_{n,d}$.
\begin{theorem}\label{th3}
Let $\e>0$ be an arbitrary constant.
\begin{description}
\item[(a)] If $\a$ is a constant satisfying the conditions of Theorem \ref{th1} or Theorem \ref{th2} where appropriate
and $d$ is sufficiently large and
$d\leq n^{1/3-\e}$ then w.h.p.
$$\chi_g(G_{n,d})\geq \frac{\a d}{\ln d}\, . $$
\item[(b)] If $\a$ is a constant satisfying the conditions of Theorem \ref{th1} or Theorem \ref{th2} where appropriate
and $d$ is sufficiently large and
$d\leq n^{1/3-\e}$ then w.h.p.
\[\chi_g(G_{n,d})\leq \frac{\a d}{\ln d}.\]
\end{description}
\end{theorem}
It is known by the result of
Frieze and {\L}uczak \cite{FL} that w.h.p.\
$\chi(G_{n,d})=(1+o(1))\frac{d}{2\ln d}$. (Of course stronger results are now known,
see Achlioptas and Moore \cite{AM} and Kemkes, P\'eres-Gim\'enez and Wormald \cite{KPW}).

Theorem \ref{th3} says nothing about $\chi_g(G_{n,d})$ when $d$ is small. We have been able to prove
\begin{theorem}\label{th4}
If $d=3$ then w.h.p.\ $\chi(G_{n,d})=4$.
\end{theorem}
It is easy to see via Brooks' theorem that w.h.p.\ the chromatic number of a random cubic graph is three
and so Theorem \ref{th4} separates $\chi$ and $\chi_g$ in this context.

We often refer to the following Chernoff-type bounds for the tails of binomial
distributions (see, e.g., \cite{AS} or \cite{JLR}). Let $X=\sum_{i=1}^n X_i$ be a sum of
independent indicator random variables such that $\Pr(X_i=1)=p_i$ and let
$p=(p_1+\cdots+p_n)/n$. Then
\begin{align}
\Pr(X\leq (1-\e)np)&\leq e^{-\e^2np/2},\label{chl}\\
\Pr(X\geq (1+\e)np)&\leq e^{-\e^2np/3},\qquad\e\leq 1,\label{chu1}\\
\Pr(X\geq \m np)&\leq (e/\m)^{\m np}.\label{chu2}
\end{align}
\subsection{Outline of the paper}
Section \ref{Gnp} is devoted to the proof of Theorem \ref{th2}. In Section \ref{Gnplow}, we prove a lower bound on
$\chi_g(G_{n,p})$ by giving a strategy for player B.
Basically, B's strategy is to follow A coloring a vertex with color $i$ by coloring a random vertex $v$ with color $i$.
Of course we mean here that $v$ is randomly chosen from vertices outside of the neighborhood of the set of vertices
of color $i$. Why does this work? Well, it is known that choosing an independent set via a greedy algorithm will
w.h.p.\ find an independent set that is about one half the size of the largest independent set. What we show is
that choosing randomly half the time also has a deleterious effect on the size of the independent set (color class)
selected. This leads to the game chromatic number being significantly larger than the chromatic number.

In Section \ref{Gnphigh}, we prove an upper bound on $\chi_g(G_{n,p})$ by giving a strategy for player A.
Here A follows the same strategy used in the proof of Theorem \ref{th1}(b), up until close to the end.
We then let A follow a more sophisticated strategy. A's initial strategy is to choose a vertex with as few
``available'' colors and color it with any available color i.e. one that does not conflict with its colored
neighbors. At a certain point there are few uncolored vertices and they all have a substantial number of
available colors. We show that the edges of the graph induced by these vertices can be partitioned into a forest
$F$ plus a low degree subgraph. Using the tree coloring strategy described in \cite{R1} we see that the low
degree subgraph does not prevent $G$ from being colored.

Having proved Theorem \ref{th2} we transfer the results to random $d$-regular graphs ($d\leq n^{1/4}$)
by showing that the underlying
structural lemmas remain true or trivially modified. This is done in Section \ref{Gnd}.

In Section \ref{Gnd} we show how to convert Theorem \ref{th1} into a random regular
graph setting. The two ranges 
$d_0\leq d\leq n^{1/4}$ and $n^{1/4}<d\leq n^{1/3-\e}$ are treated seperately. The lower range is treated in Section
\ref{Gnd1} and the upper range is treated in Section \ref{usingKV} using the ``Sandwiching Theorem'' of Kim and Vu \cite{KV}. 

In Section \ref{Gnd=3} we provide a strategy for B showing that w.h.p.\ $\chi_g(G_{n,3}) = 4$.
This proves Theorem~\ref{th4}. B's strategy is based on his ability to force A into playing on a small
set of vertices. B will then make a sequence of such forcing moves along a cycle to create a double threat and
win the game.

\section{Theorem \ref{th2}: $G_{n,p},p=d/n$}\label{Gnp}
\subsection{The lower bound} \label{Gnplow}
Let $D=\frac{d}{\ln d}$ and suppose that there are $k=\a D$ colors.
At any stage, let $C_i$ be the set of vertices that have been colored $i$ and let $C=\bigcup_{i=1}^kC_i$.
Let $U=[n]\setminus C$ be the set of uncolored vertices and let $U_i=U\setminus N(C_i)$.
Note that $[n]=\set{1,2,\ldots,n}$ is the vertex set of $G_{n,p}$.

B's strategy will be to choose the same color $i$ that A just chose and then to assign color $i$ to a random vertex in
$U_i$. The idea being that making random choices when constructing an independent set (color class) tends
to only get one of half the maximum size. A could be making better choices and so we do not manage to prove that
we need twice as many colors as the chromatic number.

Suppose that we run this process for $\th n$ rounds and that $|C_i|=2\b n/D$ where we will
later take $\th=7\a/8<1/2$ and $\b=1/2$. Let $S_i$ be the set of $\b n/D$
vertices in $C_i$ that were colored by B. We consider the probability that there exists a set $T$ of
size $\g n/D$ such that $C_i\cup T$ is independent. For expressions $X,Y$ we use the notation $X\leb Y$
in place of $X=O(Y)$ when the
bracketing is ``ugly''.
\begin{align}
\Pr(\exists C_i,T)&\leb k\binom{n}{\b n/D}\binom{n}{\g n/D}\sum_{|S|=\b n/D}
\Pr(S_i=S)(1-p)^{(2\b+\g)^2n^2/2D^2}\label{f0}\\
&\leq k\binom{n}{\b n/D}\binom{n}{\g n/D}\sum_{|S|=\b n/D}(\b n/D)!\prod_{j=1}^{\b n/D}\frac{7}{(1-p)^{2j}(1-2\th)n}
(1-p)^{(2\b+\g)^2n^2/2D^2}\label{f1}\\
&\leq k\binom{n}{\b n/D}^2\binom{n}{\g n/D}\frac{(\b n/D)!}{((1-2\th)n)^{\b n/D}}\frac{7^{\b n/D}}{(1-p)^{\b^2n^2/D^2}}
(1-p)^{(2\b+\g)^2n^2/2D^2}\nonumber\\
&\leq k\brac{\bfrac{eD}{\b}^{\b}\cdot \bfrac{7}{1-2\th}^\b\cdot\bfrac{eD}{\g}^{\g}\cdot
\exp\set{(2\b^2-(2\b+\g)^2)d/2D}}^{n/D}\nonumber\\
&=k\exp\set{(\b+\g+\b^2-(2\b+\g)^2/2+o_d(1)))d^{-1}n\ln^2 d}\label{f1a}\\
&=o(1)\nonumber
\end{align}
if $(2\b+\g)^2>2(\b+\b^2+\g)$. This is satisfied when $\b=1/2$ and $\g=3/4$. We will justify \eqref{f0} and \eqref{f1}
momentarily.

If the event $\set{\exists C_i,T}$ does not occur then
because no color class has size greater than $(2\b+\g)n/D$
the number $\ell$ of colors $i$ for which $|S_i|\geq \b n/D$ by this time satisfies
$$\frac{\ell(2\b+\g)}{D}+\frac{2(k-\ell)\b}{ D}\geq 2\th.$$
We choose $\th=7\a/8$. Since $k\geq \ell$,
this implies that
$$\frac{k}{D}\geq \frac{2\th}{2\b+\g}=\a.$$
This completes the proof of Part (a) of Theorem \ref{th2}.

{\bf Justifying \eqref{f0}:} Here we are taking the union bound over all $\binom{n}{\b n/D}\binom{n}{\g n/D}$ possible
choices of $C_i\setminus S_i$ and $T$.
In some sense we are allowing player A to simultaneously choose all possible
sets of size $\b n/D$ for $C_i\setminus S_i$. The union bound shows that w.h.p.\ all choices fail.
We do not sum over orderings of $C_i\setminus S_i$. We instead compute an upper bound on $\Pr(S_i=S)$ that holds
regardless of the order in which A plays. We consider the situation after $\theta n$ rounds.
That is, we think of the following random process: pick a graph $G \sim G(n,p)$, let Alice play the coloring game
on $G$ with $k$ colors against a player who randomly chooses an available vertex to be colored by the same color as
Alice. Stop after $\theta n$ moves. At this point Alice played with color $i$ and there are $\b n/D$ vertices that were colored
$i$ by Alice and the same number that were colored $i$ by Bob. We bound the probability that at this point there are $\g n/D$ vertices
that form an independent set with the $i$'th color class. We take a union bound over all the possible sets for Alice's vertices
and for the vertices in $T$. The probability of Bob choosing a certain set is computed below.

{\bf Justifying \eqref{f1}:}
For this we first consider a sequence of random variables
$$X_1=N=(1-2\th)n, X_j=\Bin(X_{j-1},q)\text{ where }q=(1-p)^2\text{ and }1\leq j\leq t.$$
$X_j$ is a lower bound for the number of vertices Bob can color $i$. The probability that a vertex was $i$-available at
time $j-1$ and is still $i$-available now is $(1-p)^2$. This is because two more vertices have been colored $i$.
Also, we take $X_1=N$ as a lower bound on the number of choices at the start of the process.
Then we
estimate $\E(Y_t)$ where
$$Y_t=\begin{cases}
       0&X_t=0\\\frac{1}{X_1X_2\cdots X_t}&X_t>0
      \end{cases}
$$
We use $Y_{\b Dn}$ as an upper bound on the probability that B's sequence of choices is\\
$x_1,x_2,\ldots,x_{\b n/D}$ where $S=\set{x_1,x_2,\ldots,x_{\b n/D}}$. The term $X_j$ lower bounds
the number of choices that B has and so $1/X_j$ upper bounds the probability that B chooses $x_j$. We take the expectation
of the product of these bounds over $G_{n,p}$.

Now if $B=Bin(\n,q)$ and we take $\prod_{i=1}^k\frac{1}{B+i-1}=0$ when $B=0$ then
\begin{align}
\E\brac{\prod_{i=1}^k\frac{1}{B+i-1}}&=
\sum_{\ell=1}^\n\prod_{i=1}^k\frac{1}{\ell+i-1}\binom{\n}{\ell}q^\ell(1-q)^{\n-\ell} \nonumber\\
&=\frac{1}{q^k}\prod_{i=1}^k\frac{1}{\n+i}\sum_{\ell=1}^{\n} \frac{\ell+k}{\ell}\binom{\n+k}{\ell+k}q^{\ell+k}(1-q)^{\n-\ell}\nonumber \\
&\leq\brac{\frac{1}{q^k}\prod_{i=1}^k\frac{1}{\n+i}}\brac{1+\sum_{\ell=1}^{\n} \frac{k}{\ell}\binom{\n+k}{\ell+k}q^{\ell+k}(1-q)^{\n-\ell} }.\label{pros}
\end{align}
Suppose now that $q=1-o(1)$. Then
\begin{align*}
\sum_{\ell=1}^{\n} \frac{k}{\ell}\binom{\n+k}{\ell+k}q^{\ell+k}(1-q)^{\n-\ell} &\leq 
\sum_{\ell=1}^{k/2} \frac{k}{\ell}\binom{\n+k}{\ell+k}q^{\ell+k}(1-q)^{\n-\ell}
+2\sum_{\ell=1}^{\n} \binom{\n+k}{\ell+k}q^{\ell+k}(1-q)^{\n-\ell}\\
&\leq ke^{-(\n+k)/10}+2\\
&\leq 6.
\end{align*}
Going back to \eqref{pros} we see that 
$$\E\brac{\prod_{i=1}^k\frac{1}{B+i-1}}\leq \frac{7}{q^k}\prod_{i=1}^k\frac{1}{\n+i}.$$
It follows that
\begin{align*}
&\E\bfrac{1}{X_1\cdots X_t}\\
\leq &\E\bfrac{7}{X_1\cdots X_{t-1}(X_{t-1}+1)q}\\
\leq &\E\bfrac{7^2}{X_1\cdots X_{t-2}(X_{t-2}+1)(X_{t-2}+2)q^{1+2}}\\
\vdots\\
\leq & \frac{7^t}{N(N+1)\cdots(N+t)q^{1+2+\cdots+t}}.
\end{align*}
\subsection{The upper bound}\label{Gnphigh}
We begin by proving some simple structural properties of $G_{n,p}$.
\begin{lemma}\label{l1}
If $\th>1$ and
\beq{eq5}
\bfrac{\s ed}{2\th}^\th\leq \frac{\s}{2e}
\eeq
then w.h.p.\
there does not exist $S\subseteq [n],|S|\leq \s n$ such that $e(S)\geq \th |S|$.
\end{lemma}
\proofstart
\begin{align}
\Pr(\exists S: |S|\leq \s n\text{ and }e(S)\geq \th |S|)&\leq \sum_{s=2\th}^{\s n}\binom{n}{s}
\binom{\binom{s}{2}}{\th s}\bfrac{d}{n}^{\th s}\label{3}\\
&\leq \sum_{s=2\th}^{\s n}\brac{\frac{ne}{s}\bfrac{eds}{2\th n}^{\th}}^s\nonumber\\
&=\sum_{s=2\th}^{\s n}\brac{e\bfrac{s}{n}^{\th-1}\bfrac{ed}{2\th}^{\th}}^s\label{2}\\
&=O\bfrac{d^\th}{n^{\th-1}}=o(1)\nonumber
\end{align}
provided $d=o(n^{1-1/\th})$.
\proofend

We will apply this lemma with $\th\geq 2-\e$ for $\e\ll1$ and this fits with our bound on $d$.
\begin{lemma}\label{l2}
Let $\s,\th$ be as in Lemma \ref{l1}. If $(\D-2\th)\t>1$ and
$$\bfrac{\s ed}{(\D-2\th)\t}^{(\D-2\th)\t}\leq\frac{\s}{4e}$$
then w.h.p.\ there do not exist $S\supseteq T$ such that
$|S|=s\leq \s n,|T|\geq \t s$ and $d_S(v)\geq \D$ for every $v\in T$.
\end{lemma}
\proofstart
In the light of Lemma \ref{l1}, the assumptions imply that w.h.p.\ $|e(T:S\setminus T)|\geq (\D-2\th)\t s$. In which case,
\begin{align}
&\Pr(\exists S\supseteq T,\,|S|\leq \s n,\,|T|\geq \t s: |e(T:S\setminus T)|\geq (\D-2\th)\t s)\nonumber\\
&\leq\sum_{s=2\th}^{\s n}\sum_{t=\t s}^s\binom{n}{s}\binom{s}{t}
\bfrac{edt}{(\D-2\th)\t n}^{(\D-2\th)\t s}\label{5}\\
&\leq \sum_{s=2\th}^{\s n}\sum_{t=\t s}^s\bfrac{ne}{s}^s\cdot 2^s\cdot
\bfrac{eds}{(\D-2\th)\t n}^{(\D-2\th)\t s}\nonumber\\
&=\sum_{s=2\th}^{\s n}\sum_{t=\t s}^s\brac{\frac{2ne}{s}\cdot
\bfrac{eds}{(\D-2\th)\t n}^{(\D-2\th)\t }}^s\nonumber\\
&=\sum_{s=2\th}^{\s n}\sum_{t=\t s}^s\brac{2e\bfrac{s}{n}^{(\D-2\th)\t -1}\cdot
\bfrac{ed}{(\D-2\th)\t}^{(\D-2\th)\t }}^s\label{4}\\
&=O\bfrac{d^{(\D-2\th)\t}}{n^{(\D-2\th)\t-1}}=o(1).\nonumber
\end{align}
\proofend

We will apply this lemma with $(\D-2\th)\t\geq 2$ and this fits with our bound on $d$.

Fix $\a>12$ and let
$$k=\frac{\a d}{\ln d}\text{ and }\b=\frac{\a d^{1-1/\a}}{\ln d}\text{ and }\g=\frac{16\ln^2d}{\a d^{1-1/\a}}.$$
We will now argue that w.h.p.\ A can win the game if $k$ colors are available.

A's initial strategy will be the same as that described in \cite{BFS}. Let $\cC=(C_1,C_2,\ldots,C_k)$ be a collection of pairwise disjoint subsets of $[n]$, i.e. a (partial) coloring. Let $\bigcup\cC$ denote $\bigcup_{i=1}^k C_i$. For a vertex $v$ let 
\[A(v,\cC)= \set{i\in [k]:\;v\mbox{ is not adjacent to any vertex of }C_i}, \]
and set
$$a(v,\cC)=|A(v,\cC)|.$$
Note that \( A(v, \cC) \) is the set of colors that are available at vertex \(v\) when the partial coloring is given by the sets in \( \cC \) and \( v \not\in \bigcup \cC\). A's initial strategy can now be easily defined. Given the current color classes $\cC$, A chooses an uncolored vertex $v$ with the smallest value of $a(v,\cC)$ and colors it by any available color.

As the game evolves, we let \(u\) denote the number of uncolored vertices in the graph.  So, we think of \(u\) as running ``backward'' from \( n \) to \(0\).

We show next that w.h.p.\ every $k$-coloring (proper or improper) of the full vertex set has the property that there are at most $\g n$ vertices with less than $\b/2$ available colors. Let
$$B(\cC)=\set{v:\;a(v,\cC)< 3\b/4}.$$
\begin{lemma}\label{lem5}
W.h.p., for all collections $\cC$,
$$|B(\cC)|\leq \g n.$$
\end{lemma}
\proofstart
We first note that if $|S|=\g n$ then w.h.p. $S$ contains at most $4\g^2dn$ edges. This follows from Lemma \ref{l1} with $\s=\g$ and $\th=4\g d$. It follows that for any $\e>0$ that there is a set $S_1\subseteq S$ of size at least $(1-\e)\g n$ such that if $v\in S_1$ then its degree $d_S(v)$ in $S$ is at most $8\e^{-1}\g d$.

Fix $\cC$ and suppose that $v\in S_1$. Let 
$$b(v,\cC)=\card{\set{i\in [k]:\;v\mbox{ is not adjacent to any vertex of }C_i\setminus S}}.$$ 
Thus $a(v,\cC)\geq b(v,\cC)-8\e^{-1}\g d$. $b(v,\cC)$ is the sum of independent indicator variables $X_i$, where $X_i=1$ if $v$ has no neighbors in $C_i\setminus S$ in $G_{n,p}$. Then $\Pr(X_i=1)\geq (1-p)^{|C_i|}$ and since $(1-p)^t$ is a convex function of $t$ we have
\begin{eqnarray*}
\E(b(v,\cC))&\geq&\sum_{i=1}^k(1-p)^{|C_i|}\\
&\geq&k(1-p)^{(|C_1|+\cdots+|C_k|)/k}\\
&\geq& k(1-p)^{n/k}\\
&=&\b-o(\b).
\end{eqnarray*}
It follows from the Chernoff bound \eqref{chl} that
$$\Pr(b(v,\cC)\leq 0.51\b)\leq e^{-\b/33}.$$
Now, when $\cC$ is fixed, the events $\set{b(v,\cC)\leq 0.51\b},v\in S_1$ are independent. Thus, because $a(v,\cC)\leq \b/2$ implies that $b(v,\cC)\leq 0.51\b$ we have 
\begin{align}
&\Pr(\exists\cC:|B(\cC)|\geq \g n)\nonumber\\
&\leq k^n\binom{n}{(1-\e)\g n}e^{-(1-\e)\g\b n/33}\nonumber\\
&\leq d^n\brac{\frac{e}{(1-\e)\g}\exp\set{-\frac{\a d^{1-1/\a}}{33\ln d}}}^{(1-\e)\g n}\label{1}\\
&=\exp\set{n\brac{\ln d+(1-\e)\g\brac{\ln\bfrac{1}{1-\e}+\ln\bfrac{\a}{16}+(1-1/\a)\ln d-2\ln\ln d-\frac{\a d^{1-1/\a}}{33\ln d}}}}\nonumber\\
&=o(1),\nonumber
\end{align}
for large $d$ and small enough $\e$.
\proofend

Let $u_0$ to be the last time for which A colors a vertex with at least $\b/2$ available colors, i.e., 
$$u_0= \min\set{u:\;a(v,\cC_u)\geq 3\b/4, \mbox{ for all }v \not\in \bigcup \cC_u},$$
where $\cC_u$ denotes the collection of color classes when \(u\) vertices remain uncolored.

If $u_0$ does not exist then A will win.

It follows from Lemma \ref{lem5} that w.h.p.\ $u_0\leq 2\g n$ and that at time $u_0$, every vertex still has at least $\b/2$ available colors. Indeed, consider the final coloring $\cC^*$ in the game that would be achieved if A follows her current strategy, even if she has to improperly color an edge. Let $U=\set{v\notin \cC_{u_0}:a(v,\cC^*)< \b/2}$. Now we can assume that $|U|\leq \g n$. Because the number of colors available to a vertex decreases as vertices get colored, from $u_0$ onwards, every vertex colored by $A$ is in $U$. Therefore $u_0\leq 2\g n$.

Now let $u_1$ be the first time that there are at most $2\g n$ uncolored vertices and $a(v,\cC_u)\geq \b/2, \mbox{ for all }v \not\in \bigcup \cC_u$. By the above, w.h.p.\ $u_1 \leq
u_0$, so in particular w.h.p.\ $u_1$ exists. A can determine $u_1$ but not $u_0$, as $u_0$ depends on the future.

A will follow a more sophisticated strategy from $u_1$ onwards. We will show next that we can find a sequence $U=U_0\supseteq U_1\supseteq \cdots \supseteq U_\ell$ with the following properties: The edges of $U_i:(U_{i-1}\setminus U_i)$ between $U_i$ and $U_{i-1}\setminus U_i$ will be divided into two classes, {\em heavy} and {\em light}. Vertex $w$ is a heavy (resp. light) neighbor of vertex $v$ if
the edge $(v,w)$ is heavy (resp. light).
\begin{enumerate}[{\bf (P1)}]
\item Each vertex of $U_i\setminus U_{i+1}$ has at most one light neighbor in $U_{i+1}$, for $0\leq i<\ell$.
\item All $U_i:(U_{i-1}\setminus U_i)$ edges are light for $i\geq 2$.
\item Each vertex of $U_1$ either has degree at most $\b/3$ in $U_0$ or it has at most $\b/20$ heavy neighbors in $U_0\setminus U_1$.
\item $d_{U_i}(v)\leq \b/3$ for $v\in U_i\setminus U_{i+1}$.
\item $U_\ell$ contains at most one cycle.
\end{enumerate}
From this, we can deduce that the edges of $U_0$ can be divided up into the heavy edges $E_H$, light edges $E_L$, the edges inside $U_\ell$ and the rest of the edges. Assume first that $U_\ell$ does not contain a cycle. $F=(U,E_L)$ is a forest and the strategy in \cite{R1} can be applied. When attempting to color a vertex $v$ of $F$, there are never more than three $F$-neighbors of $v$ that have been colored. Since there are at most $\b/3+\b/20$ non-$F$ neighbors, A will succeed since she has an initial list of size $\b/2$.

If $U_\ell$ contains a cycle $C$ then A can begin by coloring a vertex of $C$. This puts A one move behind in the tree coloring strategy, in which case we can bound the number of $F$-neighbors by four.

It only remains to prove that the construction P1--P5 exists w.h.p. Remember that $d$ is sufficiently large here.

We can assume without loss of generality that $|U_0|=2\g n$. This will not decrease the sizes of
the sets $a(v,U_0)$.
\subsubsection{The verification of P1--P4: Constructing $U_1$}
Applying Lemma \ref{l2} with
$$\s=2\g\text{ and }\th=d^{1/\a}\ln^3d\text{ and }\D=2\th+\b/4<\b/3\text{ and }\t=\th/\b$$
we see that w.h.p.
$$U_{1,a}'=\set{v\in U_0:d_{U_0}(v)\geq 2\th+\b/4}\text{ satisfies }
|U_{1,a}'|\leq 2\t\g n=\frac{64d^{3/\a}\ln^6d}{\a^2d^2}n.$$
We then let $U_{1,a}\supseteq U_{1,a}'$ be the subset of $U_0$ consisting of the vertices with the $2\t\g n$
largest
values of $d_{U_0}$.

We then construct
$U_{1,b}\supseteq U_{1,a}$ by repeatedly adding vertices $x_1,x_2,\ldots,x_r$ of $U\setminus U_{1,a}$ such that $x_j$
is the lowest numbered vertex not in $X_j=U_{1,a}\cup \set{x_1,x_2,\ldots,x_{j-1}}$  having at least three neighbors in
$X_j$. This ends
with $r\leq 5|U_{1,a}|$ in order that we do not violate the conclusion of Lemma \ref{l1} with
$$\s=12\t\g=\frac{192d^{3/\a}\ln^6d}{\a^2d^2}\text{ and }\th=5/2$$
which is applicable since
$$\bfrac{384ed^{3/\a}\ln^6d}{5\a^2 d}^{5/2}<\frac{192d^{3/\a}\ln^6d}{2e\a^2d^2}.$$
It follows that
\beq{U1b}
|U_{1,b}|\leq 12\t\g n.
\eeq
Next let $A$ be the set of vertices in $U\setminus U_{1,b}$ that have two neighbors in $U_{1,b}$
and let $B$ be the set of vertices in $U_{1,b}$ that have more than $\b/20$ neighbors in
$A$. We argue that w.h.p. 
\beq{sizeB}
|B|\leq \frac{1200d^{3/\a}\t\g n}{\b}=\frac{48000\ln^4d}{\a^3d^{3-6/\a}}n.
\eeq
We first consider the size of $A$. We first prove that w.h.p. 
\beq{sizeA}
|A|\leq 12d^{3/\a}\t\g n.
\eeq
For a set $S$, let $D_2(S)$ denote the set of vertices not in $S$ that have at least two neighbors in $S$.
\begin{lemma}\label{D2S}
W.h.p. $|D_2(S)|\leq d^{3/\a}|S|$ for all $|S|\leq 12\t\g n$.
\end{lemma}
\proofstart
For easy reference we note that $\t\g=\frac{40\ln^3d}{\a^2d^{2-2/\a}}$. For $|S|\leq 12\t\g n$ and $K=d^{3/\a}$, we have
\begin{align*}
\Pr(\exists |S|\leq 12\t\g:\;|D_2(S)\geq K|S|)&\leq \sum_{s=2}^{12\t\g n}\binom{n}{s}\binom{n}{K s}\brac{\binom{s}{2}\frac{d^2}{n^2}}^{Ks}\\
&\leq \sum_{s=2}^{12\t\g n}\bfrac{ne}{s}^s\bfrac{ne}{Ks}^{Ks}\bfrac{s^2d^2}{2n^2}^{Ks}\\
&=\sum_{s=2}^{12\t\g n}\brac{\bfrac{sed^2}{2Kn}^{K-1}\cdot \frac{e^{2}d^{2}}{2K}}^s\\
&=o(1).
\end{align*}
\proofend

Equation \eqref{sizeA} follows immediately from \eqref{U1b} and Lemma \ref{D2S}. To bound the size of $B$ we prove the following lemma.
\begin{lemma}\label{lemB}
W.h.p. there do not exist disjoint sets $S,T$ such that $|T|\leq t_0=12d^{3/\a}\t\g n$ and $|S|\geq 100|T|/\b$ such that each $v\in S$ has at least $\b/20$ neighbors in $T$.
\end{lemma}
\proofstart
We have
\begin{align*}
\Pr(\exists S,T\text{ denying lemma})&\leq \sum_{t=\b/20}^{t_0}\binom{n}{t}\binom{n}{100t/\b}\brac{\binom{t}{\b/20} \bfrac{d}{n}^{\b/20}}^{100t/\b}\\
&\leq \sum_{t=\b/20}^{t_0}\bfrac{ne}{t}^t\bfrac{ne\b}{100t}^{100t/\b}\bfrac{20etd}{\b n}^{5t}\\
&=\sum_{t=\b/20}^{t_0}\brac{\bfrac{t}{n}^{4-100/\b}\frac{e^{100/\b+6}\b^{100/\b-5}d^{5}}{100^{100/\b}} }^t\\
&=o(1).
\end{align*}
\proofend

Equation \eqref{sizeB} follows immediately from \eqref{sizeA} and Lemma \ref{lemB}.

Now let $A_1$ be the set of vertices in $A$ that have two neighbors of $B$. It follows from Lemma \ref{D2S} that w.h.p.
$$|A_1|\leq d^{3/\a}|B|\leq \frac{48000\ln^4d}{\a^3d^{3-9/\a}}n.$$
Next let $B_1$ be the set of vertices in $B$ that have at least $\b/20$ neighbors of $A_1$. It follows from Lemma \ref{lemB} that w.h.p.
$$|B_1|\leq \frac{100|A_1|}{\b}\leq \frac{4800000\ln^5d}{\a^4d^{4-10/\a}}n.$$
\begin{lemma}\label{A2}
W.h.p. if $s_0=e^{-d}n\leq |S|\leq s_1=\frac{4800000\ln^5d}{\a^4d^{4-10/\a}}n$ then $|N(S)|\leq d^{1+1/\a}|S|$.
\end{lemma}
\proofstart
\begin{align*}
\Pr(\exists S:\text{ denying lemma})&\leq \sum_{s=s_0}^{s_1}\binom{n}{s}\binom{n}{d^{1+1/\a}s}\bfrac{sd}{n}^{d^{1+1/\a}s}\\
&=\sum_{s=s_0}^{s_1}\brac{\frac{ne}{s}\bfrac{e}{d^{1/\a}}^{d^{1+1/\a}}}^s\\
&=o(1).
\end{align*}
\proofend

Now let $A_2$ be the set of vertices in $A_1$ that have a neighbor in $B_1$. It follows from Lemma \ref{A2} that w.h.p. 
$$|A_2|\leq \frac{4800000\ln^5d}{\a^4d^{3-11/\a}}n.$$
(Note that if $|S|<e^{-d}n$ then Lemma \ref{A2} implies that $|N(S)|\leq |S|+d^{1+1/\a}e^{-d}n$.)

Next let $U_{1,c}=U_{1,b}$ and $Y_0=A_2\cup B_1$.  We now construct $U_1\supseteq U_{1,c}$ by repeatedly adding vertices $y_1,y_2,\ldots,y_s$ of $U\setminus U_{1,c}$ such that $y_j$ is the lowest numbered vertex not in $Y_j=Y_0\cup \set{y_1,y_2,\ldots,y_{j-1}}$ that has at least two neighbors in $Y_{j}$. Taking $\th=3/2$ and using Lemma \ref{l1}, this ends with $s\leq 3|B|$ by the same argument used to show $s\leq 5|U_{1,a}|$ above. Note that
$$|U_1|\leq \g_1n=13\t\g n=\frac{208\ln^6d}{\a^2d^{2-3/\a}}.$$
We let $W=U_0\setminus U_1$ and partition the edges $W:U_1$ into light and heavy edges. 
\begin{enumerate}[{\bf (A)}]
\item $W:(Y_s\setminus U_{1,b})$: These edges will be heavy.
\item $(W\setminus (A\cup Y_s)):U_{1,b}$. These edges will be light. Note that each $v\in W\setminus (A\cup Y_s)$ has at most one neighbor in $U_{1,b}$.
\item $(A\setminus (A_1\cup Y_s)):(U_{1,b}\setminus Y_s)$. If $v\in A\setminus (A_1\cup Y_s)$ then $v$ has at most two neighbors in $U_{1,b}$. At most one of these can be in $B$ and we make the corresponding edge light. If $v$ has a neighbor in $U_{1,b}\setminus B$ then we make the corresponding edge heavy.
\item $(A_1\setminus Y_s):(U_{1,b}\setminus Y_s)$. If $v\in A_1\setminus Y_s$ then $v$ has at most two neighbors in $B$. None of these can be in $B_1$ and we make the corresponding edges heavy.
\end{enumerate}
We now have to check that {\bf P1--P4} hold. 

First consider the light edges. There is at most one for each $v\in W$ and so {\bf P1} holds.

Now consider the heavy edges. Vertices in $B$ can only have light neighbors in $W\setminus A$ and vertices in $B\setminus B_1$ have at most $\b/20$ heavy neighbors in $A$. Vertices in $U_{1,b}\setminus B$ can have at most $\b/20$ heavy neighbors in $W$. Vertices in $U_1\setminus U_{1,b}$ have (heavy) degree at most $\b/3$ in $U_0$. This verifies {\bf P3} and {\bf P4} holds by the definition of $U_{1,a}$.
\subsubsection{The verification of P1--P4: Constructing $U_2$}
Applying Lemma \ref{l2} again, with
$$\s=\g_1\text{ and }\th=3\text{ and }\D=2\th+\b/3\text{ and }\t=12/\b$$
we see that w.h.p.
\beq{eq0}
U_2'=\set{v\in U_1:d_{U_1}(v)\geq 6+\b/3}\text{ satsifies }|U_2'|\leq \g_2'=\frac{12\g_1}{\b}
\leq \frac{2500\ln^7d}{\a^3d^{3-4/\a}}.
\eeq
We then construct
$U_2\supseteq U_2'$ by repeatedly adding vertices $x_1,x_2,\ldots,x_r$ of $U_1\setminus U_2'$ such that $x_i$
has at least two neighbors in
$U_2'\cup \set{x_1,x_2,\ldots,x_{i-1}}$. This ends
with $r\leq 7|U_2'|$ in order that we do not violate the conclusion of Lemma \ref{l1} with
\beq{eq1}
\s=8\g_2'\leq\frac{20000\ln^7d}{\a^2 d^{3-4/\a}}\text{ and }\th=\frac{15}{8}
\eeq
which is applicable since
$$\bfrac{20000\cdot 4\cdot e\ln^4d}{15\a^3 d^{2-4/\a}}^{15/8}<\frac{20000\ln^7d}{2e\a^3 d^{3-4/\a}}.$$
This verifies P1--P4 with $i=1$.
\subsubsection{The verification of P1--P5: Constructing $U_i,\,i\geq 3$}
We now repeat the argument to create the sequence $U_0\supseteq U_1\supseteq \cdots \supseteq U_\ell$.
The value of $\th$ has decreased to 15/8 (see \eqref{eq1}) and
 $|U_i|\leq (12\g/\b)|U_{i-1}|$, as in \eqref{eq0}. We choose $\ell$ so that $|U_\ell|\leq \ln n$.
We can easily prove that w.h.p.\ $S$ contains at most $|S|$ edges whenever $|S|\leq \ln n$, implying P5.

This completes the proof of Part (b) of Theorem \ref{th2}.

\section{Theorem \ref{th3}: $G_{n,d}$} \label{Gnd}
We will not change A or B's strategies. We will simply transfer the
relevant structural results from $G_{n,d/n}$ to $G_{n,d}$. Some of
the unimportant constants will change, but this will not change the
verification of the success of the various strategies. We will first
do this using Theorem \ref{th2} under the assumption that $d\leq
n^{1/4}$. For larger $d$ we will use Theorem \ref{th1} and the
``sandwiching theorem'' of Kim and Vu \cite{KV}. This latter
analysis is given in Section \ref{usingKV}.
\subsection{$d_0\leq d\leq n^{1/4}$}\label{Gnd1}
Here we assume that $d_0$ is a sufficiently large constant.
We begin with the configuration model of Bollob\'as \cite{Bollcon}.
We have a set $W$ of {\em points} and this is partitioned into sets $W_1,W_2,\ldots,W_n$ of size $d$.
We define $\f:W\to[n]$ by $\f(x)=j$ for all $x\in W_j$.
We associate each {\em pairing} or {\em configuration} $F$ of $W$ into $|W|/2$ pairs
to a multigraph $G_F$ on the vertex set $[n]$. A pair $\set{x,y}\in F$ becomes an edge
$(\f(x),\f(y))$ of $G_F$.
Now there are $\frac{(dn)!}{(dn/2)!2^{dn/2}}$ pairings and
each simple $d$-regular graph (without loops or multiple edges) arises $(d!)^n$ times as $G_F$.
So for any pair of $d$-regular graphs $G_1,G_2$ we have
\beq{simple}
\Pr(G_F=G_1\mid G_F\text{ is simple})=\Pr(G_F=G_2\mid G_F\text{ is simple}).
\eeq
In order to use this, we need a bound on the probability that $G_F$ is simple.
\beq{simple0}
\Pr(G_F\text{ is simple})\geq e^{-2d^2}.
\eeq
This is the content of Lemma 2 of \cite{CFR}.

It follows from \eqref{simple} and \eqref{simple0} that for any graph property $\cA$:
\beq{cA}
\e^{2d^2}\Pr(G_F\in \cA)=o(1)\text{ implies }\Pr(G_{n,d}\in\cA)=o(1).
\eeq

We can use the above to estimate $\r=\Pr(G_{n,d/n})$ is $d$ regular. We write this as
$$\r=\Pr(G=G_{n,d/n}\text{ is $d$ regular}\mid\ \  |E(G)|=dn/2)\Pr(|E(G)|=m=dn/2).$$
It is easy to show, using Stirling's approximation,
that
$$\Pr(|E(G)|=m)=\Omega(m^{-1/2})$$
and so we concentrate on the other factor.

Let $N=\binom{n}{2}$. There are $\binom{N}{m}\leq \bfrac{Ne}{m}^m,$
graphs with vertex set $[n]$ and $m$ edges of which

$$\Omega\brac{\frac{e^{-2d^2}(dn)!}{(dn/2)!2^{dn/2}(d!)^n}}\text{ are $d$-regular.}$$

So, since $d=o(n)$,
\begin{multline}\label{eq3}
\r=\Omega\brac{\frac{e^{-2d^2}}{(dn)^{1/2}}\cdot \bfrac{dn}{e}^{dn/2}\cdot \frac{1}{(d!)^n}\cdot
\bfrac{d}{e(n-1)}^{dn/2}}=\\
\Omega\bfrac{d^{dn}}{(dn)^{1/2}e^{dn+2d^2}(d!)^n}=\Omega\brac{\bfrac{1}{10d}^{n/2}}.
\end{multline}
We need another crude estimate. We prove
a small modification of Lemma 1 from \cite{CFR}.
\begin{lemma}\label{eq6}
Given $\{a_i,b_i\},\,i=1,2,\ldots k\leq n/8d$ then
$$\Pr((a_i,b_i)\in E(G_{n,d}),\,1\leq i\leq k)\leq \bfrac{20d}{n}^k.$$
\end{lemma}
\proofstart
Let $\cG_d$ denote the set of $d$-regular graphs with vertex set $[n]$.
For $0\leq t\leq k$ we let
$$\Om_t=\{G\in \cG_d:\;\{a_i,b_i\}\in E(G),1\leq i\leq t\mbox{ and }\{a_i,b_i\}\notin E(G),t+1\leq i\leq k\}.$$
We consider the set $X$ of pairs $(G_1,G_2)\in \Om_t\times \Om_{t-1}$ such
that $G_2$ is obtained from $G_1$ by deleting disjoint edges
$\{a_t,b_t\},\{x_1,y_1\},\{x_2,y_2\}$ and replacing them by
$\{a_t,x_1\},\{y_1,y_2\}$, $\{b_t,x_2\}$. Given $G_1$, we can choose
$\{x_1,y_1\},\{x_2,y_2\}$ to be any ordered pair of disjoint edges
which are not incident with $\{a_1,b_1\},\ldots,\{a_k,b_k\}$ or their neighbours and such that
$\{y_1,y_2\}$ is not an edge of $G_1$.
Thus each $G_1\in \Om_1$ is in at least
$(D-(2kd^2+1))(D-(2kd^2+2))$ pairs, where $D=dn/2$. Each $G_2\in
\Om_{t-1}$ is in at most $2Dd^2$ pairs. The factor of 2 arises
because a suitable edge $\{y_1,y_2\}$ of $G_2$
has an orientation relative to the switching back to $G_1$.
It follows that
$$\frac{|\Om_t|}{|\Om_{t-1}|}\leq
\frac{2Dd^2}{(D-(2kd^2+1))(D-(2kd^2+2d+2))}\leq \frac{20d}{n}.$$
It follows that
$$\frac{|\Om_k|}{|\Om_0|+\cdots+|\Om_k|}\leq \bfrac{20d}{n}^k$$
and this implies the lemma.
\proofend
\subsubsection{The lower bound}\label{Gnd2}
Using \eqref{simple0} we can replace \eqref{f1a} by
$$e^{2d^2}\exp\set{(\b+\g+\b^2-(2\b+\g)^2/2+o_d(1)))d^{-1}n\ln^2d}=o(1)$$
for $d\leq n^{1/4}$. After this, we can argue as in the case $G_{n,p}$.
\subsubsection{The upper bound}\label{Gnd3}
We first need to prove the equivalent of Lemmas \ref{l1} and \ref{l2}.
\begin{lemma}\label{l1a}
If $1<\th\leq d^{1/6}\ln^3d$ and
\beq{eq5a}
\bfrac{10\s ed}{\th}^\th\leq \frac{\s}{2e}
\eeq
then w.h.p.\
there does not exist $S\subseteq [n],|S|\leq \s n$ such that $e(S)\geq \th |S|$.
\end{lemma}
\proofstart
$$\Pr(\exists S: |S|\leq \s n\text{ and }e(S)\geq \th |S|)\leq \sum_{s=2\th}^{\s n}\binom{n}{s}
\binom{\binom{s}{2}}{\th s}\p_s$$
where
$$\p_s=\max_{\substack{X\subseteq \binom{[s]}{2}\\|X|=\th s}}\Pr(E(G_{n,d})\supseteq X).$$
It follows from \eqref{simple0} that $\p_s\leq e^{2d^2}\bfrac{d}{n}^{\th s}$. If $d$ is small, say
$d\leq \ln^{1/3}n$ then we can see from the proof of Lemma \ref{l1} that
$$\Pr(\exists S: |S|\leq \s n\text{ and }e(S)\geq \th |S|)\leq O\brac{e^{2\ln^{2/3}n}\cdot\frac{d^\th}{n^{\th-1}}}=o(1).$$
We can therefore assume that $d\geq \ln^{1/3}n$ and then
\begin{align*}
\sum_{s=3d^2}^{\s n}\binom{n}{s}
\binom{\binom{s}{2}}{\th s}\p_s&\leq e^{2d^2}\sum_{s=3d^2}^{\s n}\binom{n}{s}
\binom{\binom{s}{2}}{\th s}\bfrac{d}{n}^{\th s}\\
&\leq e^{2d^2}\sum_{s=3d^2}^{\s n}\brac{e\bfrac{s}{n}^{\th-1}\bfrac{ed}{2\th}^{\th}}^s\\
&\leq e^{2d^2}\sum_{s=3d^2}^{\s n}2^{-s}\\
&=o(1).
\end{align*}
When $s\leq 3d^2$ we use Lemma \ref{eq6}. For this we will need to have $\th s\leq 3\th d^2\leq n/8d$.
The maximum value of $\th$ is $d^{1/6}\ln^3d$ and so the lemma can indeed be applied for $d\leq n^{1/4}$.
Assuming this, we have
\begin{align*}
\sum_{2\th}^{3d^2}\binom{n}{s}
\binom{\binom{s}{2}}{\th s}\p_s&\leq \sum_{2\th}^{3d^2}\binom{n}{s}
\binom{\binom{s}{2}}{\th s}\bfrac{20d}{n}^{\th s}\\
&\leq\sum_{2\th}^{3d^2}\brac{e\bfrac{s}{n}^{\th-1}\bfrac{20ed}{2\th}^{\th}}^s\\
&=o(1).
\end{align*}
\proofend
\begin{lemma}\label{l2a}
Let $\s,\th$ be as in Lemma \ref{l1a}. If
$$\bfrac{10\s ed}{(\D-2\th)\t}^{(\D-2\th)\t}\leq\frac{\s}{4e}$$
then w.h.p.\ there do not exist $S\supseteq T$ such that
$|S|\leq \s n,|T|\geq \t s$ and $d_S(v)\geq \D$ for $v\in T$.
\end{lemma}
\proofstart
We first argue that if $d\leq\ln^{1/3}n$ then we prove the lemma by just inflating the failure probability by $e^{2d^2}$
as we did for Lemma \ref{l1a}.

We therefore assume that $d\geq\ln^{1/3}n$ and write
\begin{multline*}
\Pr(\exists S\supseteq T,\,|S|\leq \s n,\,|T|\geq \t s,|e(T:S\setminus T)|\geq (\D-2\th)\t s)\\
\leq \sum_{s,t}\binom{n}{s}\binom{s}{t}\binom{t(s-t)}{(\D-2\th)\t s}\p_s
\end{multline*}
where now we have
$$\p_s=\max_{\substack{X\subseteq T\times (S\setminus T)\\|X|=(\D-2\th)\t s}}\Pr(E(G_{n,d})\supseteq X).$$
Using \eqref{simple0} we write
\begin{align*}
&\sum_{s=3d^2/\t}^{\s n}\sum_{t=\t s}^s\binom{n}{s}\binom{s}{t}\binom{t(s-t)}{(\D-2\th)\t s}\p_s\\
&\leq e^{2d^2}\sum_{s=3d^2/\t}^{\s n}\sum_{t=\t s}^s
\binom{n}{s}\binom{s}{t}\bfrac{edt}{(\D-2\th)\t n}^{(\D-2\th)\t s}\\
&\leq e^{2d^2}\sum_{s=3d^2/\t}^{\s n}\sum_{t=\t s}^s\brac{2e\bfrac{s}{n}^{(\D-2\th)\t-1}\cdot
\bfrac{ed}{(\D-2\th)\t}^{(\D-2\th)\t}}^s\\
&\leq e^{2d^2}\sum_{s=3d^2/\t}^{\s n}\sum_{t=\t s}^s 2^{-s}\\
&=o(1).
\end{align*}
When $s\leq 3d^2/\t$ use Lemma \ref{eq6}, with the same caveats on the value of $d$. So,
\begin{align*}
& \sum_{s=2\th}^{3d^2/\t}\sum_{t=\t s}^s\binom{n}{s}\binom{s}{t}\binom{t(s-t)}{(\D-2\th)\t s}\p_s\\
&\leq \sum_{s=2\th}^{3d^2/\t}\sum_{t=\t s}^s\binom{n}{s}\binom{s}{t}\binom{t(s-t)}{(\D-2\th)\t s}
\bfrac{20d}{n}^{(\D-2\th)\t s}\\
&= O\bfrac{d^{(\D-2\th)\t}}{n^{(\D-2\th)\t-1}}=o(1).
\end{align*}

\proofend

\begin{remark}\label{rem1}
We can estimate $\Pr(\exists\cC:|B(\cC)|\geq \g n)$ by multiplying \eqref{1} by $1/\r$ and
notice that it remains $o(1)$.
\end{remark}
After this, the proof will much the same as for $G_{n,p}$, but with a few constants being changed.
\subsection{$n^{1/4}\leq d\leq n^{1/3-\e}$}\label{usingKV}
Our approach in this section is to use the sandwiching technique developed by Kim and Vu in \cite{KV} to adapt the proof of
Theorem~\ref{th1}. In some sense it is pretty clear that given the results of \cite{KV}, it will be possible to translate
the results of \cite{BFS} to deal with large regular graphs. We will carry out the task, but our proof will be abbreviated
and rely on notation from the latter paper.

Without changing the strategy used in obtaining the lower bound, we show that
each intermediate result used to prove the theorem in \cite{BFS} continues to hold for random regular graphs
$G_{n,d}$ in the range where these can be approximated sufficiently well by random graphs $G_{n,d/n}$.

In order to get the required strength from the Kim-Vu coupling, however, we require $d = np = n^{\epsilon}$
for some $\epsilon \geq \e_0$, where $\e_0$ is a small absolute constant.
\subsubsection{Notation}

We use Theorem 2 in \cite{KV} to get a joint distribution on $(H_1, G, H_2)$: $G$ is $d$-regular, $H_1 \subseteq G$,
$H_1 \subseteq H_2$, and although $G \not\subseteq H_2$, this is almost true in a way we discuss further. The graphs
$H_1$ and $H_2$ are random graphs with edge probabilities $p_1$ and $p_2$, and by judicious choice of parameters
we can set $p_1 = p/(1 + \delta)$ and $p_2 = p(1 + \delta)$, where
$$p = \frac{d}{n} \mbox{ and } \delta = \Theta\left( \left(\frac{\ln n}{d}\right)^{1/3} \right).$$

Constants defined in \cite{BFS} are in terms of $p$ and we will make this relationship explicit. Of note is the constant
$$\ell_1(p) = \log_b n - \log_b \log_b np - 10 \log_b \ln n$$
where $b=b(p)=\frac{1}{1-p}$.
\subsubsection{Kim-Vu coupling}

The construction of the coupling $(H_1, G, H_2)$ in \cite{KV} yields $H_1 \subseteq G$ w.h.p.,
but not $G \subseteq H_2$. As a substitute for such a result, we prove the following lemma.

\begin{lemma}
\label{lemma:degree}
$\Delta(G \setminus H_2) = O(1)$ w.h.p.
\end{lemma}
\proofstart
We rely on the bound $\Delta(G \setminus H_2) \le \Delta(G) - \delta(H_2) + \Delta(H_2 \setminus G).$
Trivially, $\Delta(G) = d$. Part 3 of Theorem 2 in \cite{KV} states that w.h.p.
$$\Delta(H_2 \setminus G) \le \frac{(1 + o(1)) \ln n}{\ln(\delta d/\ln n)} =
\frac{(1 + o(1)) \ln n}{\frac{2}{3}\ln d - \frac{2}{3}\ln \ln n+O(1)} = \frac{3+o(1)}{2\epsilon}.$$
We prove that w.h.p.\ $\delta(H_2) \ge d$.
For any vertex $v$, $\deg_{H_2} v$ follows the binomial distribution $B(n-1,p_2)$. By the Chernoff bound,
$$\Pr[\deg_{H_2} v < d] \le \Pr\left[B(n-1, p_2) < \left(1 - \frac{\delta}{2(1 + \delta)}\right)(n-1) p_2 \right]
\le e^{-\delta^2/(10(1 + \delta))}.$$
We can simplify the exponent here to
$$-\frac{\delta^2 d}{10 (1 + \delta)} = - \frac{\Omega((\ln n)^{2/3} d^{1/3})}{1 + \delta} \le -\Omega(n^{\epsilon/3}).$$
So $\Pr[\deg_{H_2} v < d] \le O(n^{-\Omega(n^{\epsilon/3})})$ and $\Pr[\delta(H_2) < d]= o(1)$, completing the proof.
\proofend

\subsubsection{Bounds}

We first prove a few auxiliary bounds on the relationship between $p$, $p_1$, and $p_2$, as well as
other constants in terms of these probabilities.

\begin{bound}
\label{bound:frac1}
$$1\geq \frac{\ell_1(p)}{\ell_1(p_1)}\geq 1-2\delta \mbox{, and } 1\leq \frac{\ell_1(p_1)}{\ell_1(p)} \le 1 + 2\delta.$$
\end{bound}
\proofstart
We first note that if $np\to\infty$ then the derivative $(\log_{b(p)}np)'<0$ and so if we let
$x = \frac{n}{\log_{np}np\,\ln^{10}n}$ then,
\begin{align*}
\frac{\ell_1(p_1)}{\ell_1(p)} &= \frac{\log_{b(p_1)} n - \log_{b(p_1)}
\log_{b(p_1)} np_1 - 10 \log_{b(p_1)} \ln n}{\log_{b(p)} n - \log_{b(p)} \log_{b(p)} np - 10 \log_{b(p)} \ln n} \\
&\le  \frac{\log_{b(p_1)} n - \log_{b(p_1)} \log_{b(p)} np - 10 \log_{b(p_1)} \ln n}{\log_{b(p)}
n - \log_{b(p)} \log_{b(p)} np - 10 \log_{b(p)} \ln n} \\
&= \frac{\log_{b(p_1)}(x)}{\log_{b(p)}(x)} = \frac{\ln b(p)}{\ln b(p_1)} \leq \frac{p(1+p)}{p_1}\leq  1 + 2\delta.
\end{align*}
This proves the second inequality. For the first, we take the reciprocal, and note that $(1 + 2\delta)^{-1} > 1 - 2\delta.$
\proofend

\begin{bound}
\label{bound:frac2}
$$1\geq \frac{\ell_1(p_2)}{\ell_1(p)}\geq 1-2\delta \mbox{, and } 1\leq \frac{\ell_1(p)}{\ell_1(p_2)} \le 1 + 2\delta.$$
\end{bound}
\proofstart
Apply Bound~\ref{bound:frac1} with $p_2$ in place of $p$ and $p$ in place of $p_1$, since their relationships are the same.
\proofend

Note now that if $d=n^\th$ where $\th=\Theta(1)$ then
$$\frac{\log_b\log_bnp}{\log_bn}=\frac{\ln\log_bnp}{\ln n}\approx 1-\th$$
which implies that
\beq{logbnp}
\ell_1(p)=(\th+o(1))\log_bnp.
\eeq
\begin{bound}
\label{bound:exp}
$$(1 - p_1)^{\ell_1(p)} = \frac{\ell_1(p) (\ln n)^{10}}{(\th+o(1))n} \mbox{ and } (1 - p_2)^{\ell_1(p)}
= \frac{\ell_1(p) (\ln n)^{10}}{(\th+o(1))n}.$$
\end{bound}
\begin{proof}
It follows from Bound \ref{bound:frac1} that
$$(1 - p_1)^{\ell_1(p)}= (1+o(1))(1 - p)^{\ell_1(p)}=(1+o(1))\frac{(\log_bnp)( \log^{10}n)}{n}.$$
Now use \eqref{logbnp}. The proof for $p_2$ is similar.
\end{proof}

\subsubsection{Lemmas used for the lower bound in \cite{BFS}}
The strategy used in \cite{BFS} to prove the lower bound relies on probabilistic assumptions labeled there as
Lemmas 2.1 through 2.4. By assuming that those lemmas hold for random graphs (and occasionally referencing
the proofs of the original lemmas), we prove that they hold in the random regular case as well. It follows
that the lower bound of Theorem~\ref{th1} is valid in the case of $G_{n,d}$ as well, provided our assumption
that $d=n^\epsilon$ holds.

\begin{lemma}[Lemma 2.1 of \cite{BFS}]
For every $S \subseteq [n]$ with $|S| = \ell_1(p)$, w.h.p.\ $$\ell_1(p) (\ln n)^9 \le \left| \ol{N}(S) \right|
\le \ell_1(p) (\ln n)^{11}.$$
\end{lemma}

\begin{proof}
For an $S$ as above, $\ol{N}(S) = \ol{N}_G(S) \subseteq \ol{N}_{H_1}(S)$. The distribution of $\ol{N}_{H_1}(S)$ is
binomial with
mean $n (1 - p_1)^{\ell_1(p)}$, which is at most $O(\ell_1(p) (\ln n)^{10})$ by Bound~\ref{bound:exp}.
We can use Chernoff bounds to get $|\ol{N}_{H_1}(S)| \le \ell_1(p) (\ln n)^{11}$, which
implies the same for $|\ol{N}(S)|$.

The proof of the lower bound is similar, except that we don't have the strict containment $\ol{N}_{H_2}(S)
\subseteq \ol{N_G}(S)$.
However, by Lemma~\ref{lemma:degree}, any vertex in $S$ has $O(1)$ neighbors in $G$ that it does not have in $H_2$.
Therefore $|\ol{N}_{G}(S)| \le |\ol{N}_{H_2}(S)| + O(|S|)$. Because $|S| = \ell_1(p)$, and the Chernoff bound
gives $|N_{H_2}|=\Omega(\ell_1(p) (\ln n)^{10})$
w.h.p., this difference will be absorbed in the $(1+o(1))$ asymptotic factor.
\end{proof}

\begin{lemma}[Lemma 2.2 of \cite{BFS}]
W.h.p.\ there do not exist $S, A, B \subseteq n$ such that (conditions omitted) and every $x \in B$ has
fewer than $ap/2$ neighbors in $A$ (where $a = |A|$).
\end{lemma}

\begin{proof}
The proof of the corresponding lemma in \cite{BFS} relies on the distribution to say that the number of neighbors any
$x \in B$ has in $A$ is distributed according to the binomial distribution $B(a, p)$, and uses the Chernoff bound
$\Pr[B(a,p) \le ap/2]^{b_1} \le e^{- ab_1p /8}.$

The number of edges between $x$ and $A$ is bounded below by the number of such edges in the graph $H_1$, which is
distributed according to $B(a, p_1)$. So we replace the bound above by
$$\Pr[B(a,p_1) \le ap/2] = \Pr\left[B(a, p_1) \le \frac{a p_1(1 + \delta)}{2}\right] \le
\Pr\left[B(a, p_1) \le a p_1 \left(1 - \frac{1 - \delta}{2}\right)\right].$$
By the Chernoff bound, this is at most $e^{-ap_1(1 - \delta)^2/8} \le e^{-(1 - o(1)) a b_1 p/8},$
and the argument of \cite{BFS} still goes through.
\end{proof}

\begin{lemma}[Lemma 2.3 of \cite{BFS}]
Let $a_1 = 2000\e^{-2}$ where $\e$ is a small positive constant.
W.h.p.\ there do not exist sets of vertices $S, T_1, \dots, T_{a_1}$ such that (conditions omitted)
and $N(S) \cap T_i = \emptyset$ for $i = 1, \dots, a$.
\end{lemma}
\begin{proof}
If such sets exist in the graph $G$, then they will still exist when we lose some edges in passing to the graph $H_1$.
Examining the proof in \cite{BFS} we see that all it requires is to consider the following
factor which inflates the $o(1)$ probability estimate they use by:
$$\left( \frac{1 - p_1}{1 - p} \right)^{a_1\binom{\epsilon\ell_1(p)/21}{2}} \le
\left( \frac{1 - p_1}{1 - p} \right)^{3\ell_1(p)^2}=\brac{1+\frac{p\d}{(1+\d)(1-p)}}^{3\ell_1(p)^2}=1+o(1).$$
\end{proof}

\begin{lemma}[Lemma 2.4 of \cite{BFS}]
Let $t = \frac{n}{\ell_1(p) (\ln n)^7}.$ W.h.p.\ there do not exist pairwise disjoint sets of vertices $S_1, \dots,
S_t, U,$ such that (conditions omitted)
and $|U \cap \ol{N}(S_i)| \le \ell_1(p) (\ln n)^8$ for $i = 1, \dots, t$.
\end{lemma}
\begin{proof}
Suppose such sets exist in the graph $G$. By Lemma~\ref{lemma:degree} each vertex of $S_i$ has at most $O(1)$
neighbors in $G$ that are not in $H_2$; therefore in passing to the graph $H_2$, $|U \cap \ol{N}(S_i)|$ will be at most
$\ell_1(p) (\ln n)^8 + O(\ell_0(p))$ where each $|S_i|=\ell_0(p)=\ell_1(p) + C/p$ for some
constant $C$ (a fact we will use again). Since $\ell_1(p) = \frac{\ln n}{(\th+o(1))p}$, the size of $|U \cap
\ol{N}(S_i)|$ in $H_2$
is also $(1 + o(1))\ell_1(p) (\ln n)^8$. There is room in the argument of \cite{BFS} to prove this lemma for intersections
 of this size as well. Therefore we will proceed by arguing that w.h.p.\ sets such as $S_1, \dots, S_t, U$
do not exist in $H_2$.

The proof in \cite{BFS} hinges upon the claim that $(1 - p)^{\ell_0(p)} = \Omega( (1-p)^{\ell_1(p)}).$ So it suffices
to prove
that $(1 - p_2)^{\ell_0(p)} = \Omega( (1-p)^{\ell_1(p)})$. We split $(1 - p_2)^{\ell_0(p)}$ into
factors $(1 -p_2)^{\ell_1(p)}$ and
$(1 - p_2)^{C/p}$. By Bound~\ref{bound:exp}, the first factor is $\Omega((1-p)^{\ell_1(p)})$. The second factor is no less
than $(1 - p_2)^{C/p_2}$, which stays in $(e^{-C},4^{-C}]$ as $p_2$ ranges over $(0, 1/2]$, so it is effectively a
constant.
\end{proof}

\subsubsection{Lemmas used for the upper bound in \cite{BFS}}
As in the lower bound, the strategy used in \cite{BFS} to prove the upper bound relies on some properties
that hold w.h.p.\ in $G_{n,p}$.
Again, we apply the Kim-Vu Sandwich Theorem \cite[Theorem 2]{KV} to show that the same properties hold
in $G_{n,d}$ as well.

\newcommand{\C}{\ensuremath{\mathcal{C}}}
\newcommand{\B}{\ensuremath{\mathcal{B}}}
\newcommand{\comment}[1]{}
The first player's strategy described in \cite{BFS} is simple --- she chooses an uncolored vertex with
minimal number of available colors and then colors it with an arbitrary (available) color. We present
some notation used in \cite{BFS} during the analysis of the strategy. Given a (partial) coloring $\C$
and a vertex $v$ let $\alpha(v,\C)$ be the number of available colors for $v$ in $\C$. For a constant $\alpha > 3$ define
\[ \beta_G = \alpha\frac{n(np)^{-1/\alpha}}{\log_b np}, \qquad \qquad \gamma_G = \frac{10n\ln n}{\beta_G} \]
and
\[ \B(\C) = \{ v \mid \alpha(v) \leq \beta_G / 2\} . \]
The first lemma states that there are not many vertices with few available colors. We show that the same
is also true in $G_{n,d}$.
\begin{lemma}[Lemma 3.1 of \cite{BFS}]
W.h.p.\ for all collections $\C$,
\[ |\B(\C)|  \leq \gamma .\]
\end{lemma}
\proofstart
Here we can just use Remark \ref{rem1}.
\proofend

\begin{lemma}[Lemma 3.2 of \cite{BFS}]
W.h.p.\ every subset $S$ of $G_{n,p}$ of size $s$ spans at most $\phi = \phi(s) = (5ps + \ln n)s$ edges.
\end{lemma}
\begin{proof}
The proof in \cite{BFS} actually gives the result for $\phi_2 = (4.5ps + 0.9 \ln n)s$.
Thus, applying Lemma 3.2 of \cite{BFS} to $H_2$ gives that w.h.p.\ every set $S$ of size $s$
spans at most $(4.5p_2s + 0.9 \ln n)s$ edges. Since w.h.p.\ every vertex of $G$ touches at most
$O(1)$ edges not in $H_2$, we have that w.h.p.\ the number of edges spanned by $S$ in $G$ is bounded by
\[ \left(4.5p_2s + 0.9 \ln n + O\left(1\right)\right)s = (4.5ps(1+o(1)) +
0.9 \ln n+O(1))s \leq (5ps + \ln n)s \]
as required.
\end{proof}
This completes the proof of Theorem \ref{th3}.

\section{Theorem~\ref{th4}: Random Cubic Graphs} \label{Gnd=3}

Consider the coloring game on $G_{n,3}$ with three colors. We
describe a strategy for B that wins the game for him w.h.p., so
$\chi_g(G_{n,3}) \ge 4$ w.h.p. This proves Theorem~\ref{th4}: in general $\chi_g(G) \le \Delta(G) + 1$
where $\Delta$ denotes maximum degree, and in particular $\chi_g(G_{n,3}) \le 4$.

We proceed in two steps. First, we describe a strategy for B that wins the game, given the existence
of a subgraph $H$ in $G$ satisfying certain conditions. Next, we will prove that w.h.p.,
a random cubic graph contains such a subgraph.

\subsection{The winning strategy}

We will say that two vertices are \emph{close} if they are connected by a path of
length two or less, and that a path is \emph{short} if some vertex on it is close to both endpoints.
(This is not the same as being of length at most four). Vertices that are not close are
{\em far apart} and a path that is not short is {\em long}.
The motivation for this terminology is that coloring a vertex can only have an effect on vertices that are
close to it; we will make this precise later on.

\begin{figure}
\centerline{\includegraphics{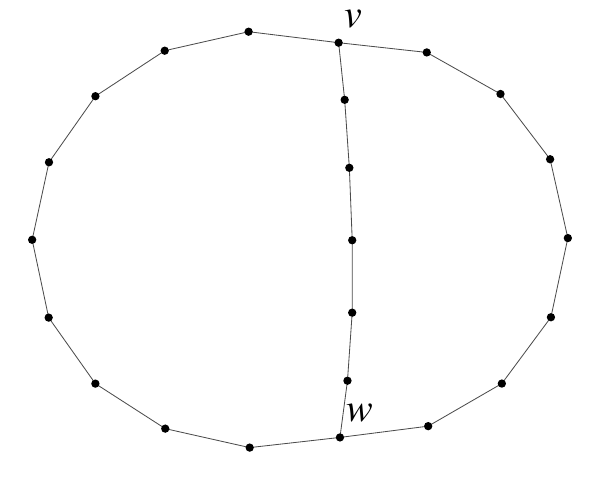}}
\caption{The subgraph $H$ required for Bob's strategy on Random Cubic Graphs.}
\label{fig:cubic1}
\end{figure}

We first assume the existence of a subgraph $H$ with the following properties (see Figure~\ref{fig:cubic1}):
\begin{enumerate}
\item $H$ consists of two vertices, $v$ and $w$, together with three (internally disjoint) paths from one to the other.
\item Each of the paths consists of an even number of edges.
\item No two vertices in $H$ are connected by a short path outside of $H$ (in particular, $H$ is induced).
\item The three paths themselves are all long.
\end{enumerate}
In addition, \\
{\bf Property F:} if A goes first, then the vertex colored by A on her first move is far from $H$.

B first plays on the vertex $v$. Provided A's next move is not on the vertex $w$, or on the
neighbors of $v$ or $w$, it is close to at most one of the three paths which make up $H$
(this follows from Properties 3 and 4). The
remaining two paths form a cycle containing $v$, with no other already colored vertices
close to the cycle; by Property 2, the cycle is even. Call the vertices around the cycle $(v, v_1, v_2, \dots, v_{2k-1})$.

Starting from this even cycle, B proceeds as follows. He colors $v_2$ a different color
from $v$; this creates the threat that on his next move, he will color the third neighbor of $v_1$
the remaining color, leaving no way to color $v_1$ and winning. We will call such a move by
B a \emph{forcing move at $v_1$}. A can counter this threat in several ways:
\begin{itemize}
\item By coloring $v_1$ the only remaining viable color.
\item By coloring $v_1$'s third neighbor the same color as either $v$ or $v_2$.
\item By coloring that neighbor's other neighbors in the color different from both $v$ and $v_2$.
\end{itemize}
In all cases, A must color some vertex close to $v_1$, that does not lie on the cycle.

B continues by making a forcing move at $v_3$: coloring $v_4$ a different color from $v_2$.
Continuing to play on the even vertices $v_{2i}$, B makes forcing moves at each odd $v_{2i-1}$.
By Property 3 of $H$, the set of vertices A must play on to counter each threat are disjoint; thus,
A's response to each forcing move does not affect the rest of the strategy. By Property F, A's
first play does not affect the strategy either.

When B colors $v_{2k-2}$, this is a forcing move both at $v_{2k-3}$ and at $v_{2k-1}$ (provided
Bob chooses a color different both from $v_{2k-4}$ and $v$). A cannot counter both threats, therefore B wins.

We now account for the remaining few cases. If A colors a neighbor of $v$ or $w$ on her second
move, this vertex will be close to all three possible even cycles. However, we know that all three paths in $H$ have
even length. Therefore we can still apply this strategy to the even cycle not containing the vertex A colored.
Even though it will be close to $v$ or $w$, we will never need to force at $v$ or at $w$, because we only force at
odd numbered vertices along the path.

Finally, if A colors $w$ itself, then there is no path we can choose that will avoid the vertex.
Instead, B picks any of the paths from $v$ to $w$, and makes forcing
moves down that path. Provided that the path is sufficiently long to do so (which follows from Property 4),
the final move will be a forcing move in two ways, winning the game for B once again.

\subsection{Proof of the existence of $H$}

It remains to show that the subgraph $H$ exists w.h.p.\ (even allowing for A's first move).
We will assume $G$ is chosen by adding a random perfect matching to a cycle $C$ on $n$ vertices, and find $H$ w.h.p.
That this is a contiguous model to $G_{n,3}$ is well known,
see \cite{Worm}. In the following, let $c$ be a constant; we will later see
that we need $c$ to be less than $1$ for the proof to hold.

We begin by counting {\em good} segments of length $m=\rdown{c \sqrt{n}}$ on $C$, by which we mean
those with no internal chords.
First of all let $X$ be twice the number of chords that intercept segments of length $m$ or less -- these are the
only chords that could possibly be internal to a segment of the desired length. $X$ can be written as
the sum $X_1 + X_2 + \cdots X_n$, where $X_i$ is the 0-1 indicator for the $i$-th vertex (call it $v_i$)
to be the endpoint of such a chord. Also, let $Y_i$ denote the length of the smaller of the two segments
defined by $v_i$ (this segment stretches from $v_i$ to its partner). Thus
$$\Pr(Y_i=t)=\begin{cases}
              \frac{2}{n-1}&2\leq t\leq \rdown{(n-1)/2}\\
              \frac{1}{n-1}&t=n/2,\ n\text{ even}
             \end{cases}
$$
Clearly $X_i = 1$ if and only if $Y_i \le m$, and so
$$\E(X_i)=\frac{2m}{n-1} \text{ and }
\E(X) =\frac{2mn}{n-1}\approx 2c\sqrt{n}.$$
In addition, $\Var(X_i) \le \E(X_i)$,
and so
$$\Var(X) = \sum_{i=1}^n \Var(X_i) +\sum_{i \ne j} \Cov(X_i, X_j) \le
\E(X) + \sum_{i \ne j} \Cov(X_i, X_j).$$
Now
\begin{align*}
\Cov(X_i, X_j)&=-\frac{4m^2}{(n-1)^2}+\sum_{t=2}^m\Pr(X_i=1\mid Y_j=t)\Pr(Y_j=t)\\
&\leq-\frac{4m^2}{(n-1)^2}+ \frac{2m}{n-3}\cdot \frac{2m}{n-1}\\
&=\frac{8m^2}{(n-1)^2(n-3)}.
\end{align*}
Thus,
$$\Var(X)\leq \E(X)+\frac{8m^2n}{(n-1)(n-3)}\approx\E(X).$$
By Chebyshev's inequality,
$$\Pr(|X - \E(X)| \le \lambda \E(X)) \le \frac{\Var(X)}{\lambda^2 \E(X)^2} \le \frac{2}{\lambda^2 c \sqrt{n}}.$$
Putting $\lambda=n^{-1/5}$ we see that w.h.p.\ $X \sim 2c\sqrt{n}$.

Consider the $n$ different segments of length $m$ on $C$. Each chord counted by $X$ eliminates
at most $m$ of these segments as being good, which leaves $(1 - c^2)n$ segments remaining. We will want non-overlapping
good segments; each good segment overlaps at most $2m$ other good segments and so we can assume that we can find
$2n_1\sim (c^{-1} - c)\sqrt{n}/2$ non-overlapping good segments w.h.p. Here the segments are $\s_1,\s_2,\ldots,\s_{2n_1}$
are in clockwise order around $C$. We pair them together $P_i=(\s_i,\s_{n_1+i}),i=1,2,\ldots,n_1$.

\begin{figure}
\centerline{\includegraphics[height=0.45\textheight]{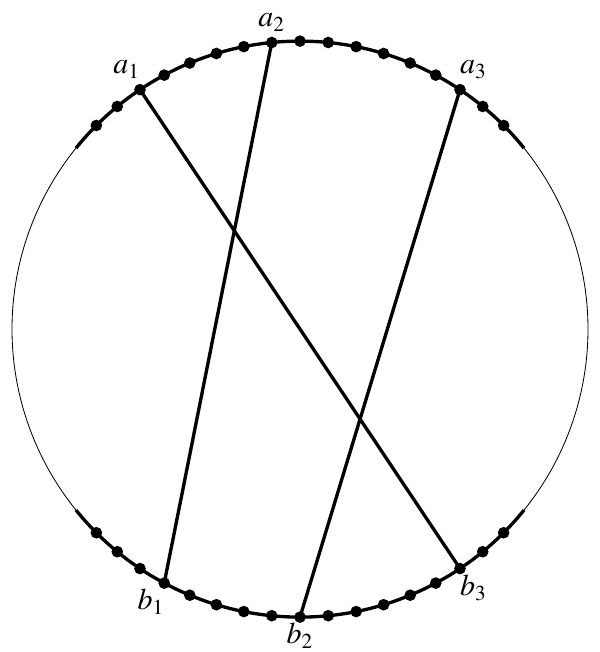}}
\caption{A typical example of the subgraph $H$ found in the Hamiltonian cycle model.}
\label{fig:cubic2}
\end{figure}

Pick any pair $P_j$. If there are exactly 3 chords from one segment to the other, as in
Figure~\ref{fig:cubic2}, then we will construct
$H$ as follows (assuming $a_i$ and $b_i$ are the endpoints of the chords, as labeled in Figure~\ref{fig:cubic2}):
\begin{itemize}
\item Set $v$ and $w$ to be $a_2$ and $b_2$, respectively.
\item The first path from $v$ to $w$ is $(a_2, \dots, a_1, b_3, \dots, b_2)$, where the vertices in the ellipses
are chosen along $C$.
\item The second path from $v$ to $w$ is $(a_2, b_1, \dots, b_2)$.
\item The third path from $v$ to $w$ is $(a_2, \dots, a_3, b_2)$.
\end{itemize}
The paths given above require that the three chords are $(a_1, b_3)$, $(a_2, b_1)$, and $(a_3, b_2)$.
%however, similar paths can be constructed provided that $(a_2, b_2)$ is not one of the chords.
In order for $H$ to satisfy
Properties 2 and 4, we impose conditions on the lengths of the paths $(a_1, \dots, a_2)$, $(a_2, \dots, a_3)$, $(b_1,
\dots, b_2)$, and $(b_2, \dots, b_3)$: they must not be too small, and must have the right parity so that the three paths
from $v$ to $w$ have even length. However, these conditions (and the condition that $(a_2, b_2)$ must not be a chord)
eliminate only a constant fraction of the possible chords; therefore there are $\Omega(m^6)$ ways to choose the chords.

The probability, then, that a subgraph $H$ can be found between two given good segments, is at least
\beq{asd}
\Omega\bfrac{m^6}{n^3} \cdot \left(1 - \frac{m}{n -2 m}\right)^{2m}
\eeq
where the last factor bounds the probability that there are no extra chords between the two segments.
This tends to a constant $\z_1$ that does not depend on $n$. Thus the expected number of $j$ for which
$P_j$ has Properties 1,2 and 4 are satisfied is at least $\z_1n_1$.

We now consider the number of pairs of good segments in which we can hope to find this structure.
In order to ensure that,
should a subgraph $H$ be found, it satisfies Property F, we eliminate all pairs which contain a vertex close
to the vertex $A$ chooses on her first move -- a constant number of pairs.

To ensure Property 3 we eliminate all pairs $P_j$ in which two
vertices have chords whose other endpoints are $1$ or $2$ edges apart.  This happens with probability
$O(1/n)$ for any two vertices, regardless of the disposition of the other chords incident with the segments
in $P_j$. The pair $P_j$s contains $\binom{2 m}{2} \le 2c^2 n$
pairs of vertices. Therefore with probability at least $\left(1 - O(1/n)\right)^{2 c^2 n}$, which tends to a constant,
$\z_2$ say, a pair $P_j$ satisfies Property 3.

Thus the expected number of $j$ for which the pair $P_j$ give rise to a copy of $H$ satisfying all required properties
is at least $\z n_1$ where $\z=\z_1\z_2$. To prove concentration for the number of $j$ we can simply use the
Chebyshev inequality. This will work, because exposing the chords incident with a particular pair $P_j$
will only have a small effect on the probability that any other $P_j'$ has the required properties.
\proofend

\end{document}